\documentclass[11pt, onesided, reqno]{amsart}

\usepackage{graphicx}
\usepackage{amsmath,amsthm}
\usepackage{amsfonts,amssymb}
\usepackage{dsfont}

\usepackage{enumerate}

\newtheorem{theorem}{Theorem}

\newtheorem{lemma}[theorem]{Lemma}

\theoremstyle{definition}

\newtheorem{remark}[theorem]{Remark}

\newcommand{\E}{\mathbb{E}}

\newcommand{\N}{\mathbb{N}}

\newcommand{\R}{\mathbb{R}}
\newcommand{\Pb}{\mathbb{P}}

\newcommand{\equi}{\mathop{\sim}\limits}
\def\={{\;\mathop{=}\limits^{\text{(law)}}\;}}

\catcode`,\active

\catcode`\,12

\allowdisplaybreaks

\begin{document}

\title{ On the supremum of products of symmetric stable processes}

\author[Christophe Profeta]{Christophe Profeta}

\address{Laboratoire de Math\'ematiques et Mod\'elisation d'\'Evry, Universit\'e d'Evry-Val d'Essonne, B\^atiment IBGBI, 23 Boulevard de France, 91037 Evry Cedex, France. {\em Email}: {\tt christophe.profeta@univ-evry.fr}}

\keywords{Persistence probability - Stable processes}

\subjclass[2010]{60G52, 60J65}

\begin{abstract} We study the asymptotics, for small and large values, of the supremum of a product of symmetric stable processes. We show in particular that the persistence exponent remains the same as for only one process, up to some logarithmic terms.
\end{abstract}

\maketitle

\section{Introduction}
For $n\in \N$, let $(Z^{(i)}, \, 1\leq i\leq n)$ be independent symmetric $\alpha$-stable L\'evy processes with $\alpha \in (0,2]$. 
In this short note, we are interested in the study of the random variable 
$$\mathcal{S}_n=\sup_{0\leq u\leq 1}  \prod_{i=1}^n Z_u^{(i)}.$$
Except when $n=1$, in which case the double Laplace transform of $\mathcal{S}_1$ is classically given by  fluctuation theory (see for instance Bertoin \cite[p.174]{Ber}),  it does not seem evident to compute explicitly the law of $\mathcal{S}_n$, and we shall rather study its asymptotics $\Pb\left(\mathcal{S}_n\geq x\right)$ as $x\rightarrow +\infty$ and $\Pb\left(\mathcal{S}_n\leq \varepsilon\right)$ as $\varepsilon\rightarrow 0$.\\

\noindent
Most of the paper is devoted to the computation of the limit as $\varepsilon\rightarrow 0$, which is  known as a persistence problem, see the surveys \cite{AS, BMS}. By scaling, this amounts to the study of the first entrance time of the $n$-dimensional stable process $(Z^{(i)},  \, 1\leq i\leq n)$ into the  "hyperbolic" domain $\mathcal{H}_n=\{(z_1, \ldots, z_n)\in \R^n,\;  \prod_{i=1}^n z_i \geq 1\}$ : 
$$\Pb\left(\mathcal{S}_n\leq \varepsilon \right) = \Pb\left(R_n> \frac{1}{\varepsilon^{\frac{\alpha}{n}}} \right)\qquad \text{where }\quad R_n=\inf\left\{u\geq0,\,  \prod_{i=1}^n Z_u^{(i)}\geq 1\right\}.$$
There are several papers in the literature dealing with entrance and exit times of symmetric stable processes, mainly for three families of domains :  cones and wedges (Ba\~nuelos and Bogdan \cite{BBCon}, M\'endez-Hern\'andez \cite{MHCon}), parabolic domains (Ba\~nuelos and Bogdan \cite{BBPar}) and unbounded convex domains (M\'endez-Hern\'andez \cite{MHUnb}).
Here, since the domain $\mathcal{H}_n$ is non-connected, not much is known regarding $R_n$ and we shall tackle the problem directly by working with $\mathcal{S}_n$.\\

\noindent
We start with the Brownian case, i.e. $\alpha=2$.
\begin{theorem}\label{theo:BM}
Let $(W^{(i)}, \, 1\leq i\leq n)$ be independent Brownian motions.  There exist two constants $0<\kappa_1 \leq \kappa_2 <+\infty$ such that the following estimates hold.
\begin{enumerate}
\item Large deviations :
$$\kappa_1\, x^{-\frac{1}{n}}  \exp\left(-\frac{n}{2}x^{\frac{2}{n}}\right) \,  \leq \Pb\left(\sup_{0\leq u\leq 1} \prod_{i=1}^n W_u^{(i)} \geq x\right)\leq \kappa_2\,x^{-\frac{1}{n}}  \exp\left(-\frac{n}{2}x^{\frac{2}{n}}\right)\, \qquad (x\rightarrow +\infty)$$
\item Persistence probability :
$$\kappa_1\, \varepsilon \leq \Pb\left(\sup_{0\leq u\leq 1} \prod_{i=1}^n W_u^{(i)} \leq \varepsilon\right) \leq \kappa_2\, \varepsilon \left|\ln(\varepsilon)\right|^{n}\qquad (\varepsilon\rightarrow 0)$$
\end{enumerate}
\end{theorem}

\noindent
In the non-Gaussian stable case, the situation is different. 
\begin{theorem}\label{theo:Levy}
Let $(L^{(i)}, \, 1\leq i\leq n)$ be independent  symmetric $\alpha$-stable L\'evy processes with $\alpha \in (0,2)$. There exist two constants $0<\kappa_1 \leq \kappa_2 <+\infty$ such that the following estimates hold.
\begin{enumerate}
\item Large deviations :
$$\kappa_1\, \frac{(\ln(x))^{n-1}}{x^\alpha} \leq \Pb\left(\sup_{0\leq u\leq 1} \prod_{i=1}^n L_u^{(i)} \geq x\right)\leq \kappa_2\,\frac{(\ln(x))^{n-1}}{x^\alpha}  \qquad (x\rightarrow +\infty)$$
\item Persistence probability :
$$\kappa_1\, \varepsilon^{\alpha/2} \leq \Pb\left(\sup_{0\leq u\leq 1} \prod_{i=1}^n L_u^{(i)} \leq \varepsilon\right)\leq \kappa_2\, \varepsilon^{\alpha/2}\left|\ln(\varepsilon)\right| \qquad (\varepsilon\rightarrow 0)$$
\end{enumerate}
\end{theorem}

\begin{remark}
The presence of extra logarithmic terms in the persistence probability of Brownian motion is due to some  additive phenomenons. Indeed, recall the estimates (see Bertoin \cite[p.219]{Ber}) :
\begin{equation}\label{eq:asymp}
\Pb(|Z_1 |\leq \varepsilon) \equi_{\varepsilon\rightarrow 0} k\, \varepsilon\qquad \text{and}\quad   \Pb\left(\sup_{0\leq u\leq 1}Z_u \leq \varepsilon\right) \equi_{\varepsilon\rightarrow 0} c\, \varepsilon^{\frac{\alpha}{2}}
\end{equation}
for some positive constants $k$ and $c$. When $\alpha<2$, the second asymptotics is the leading one, while for $\alpha=2$, 
they are of the same order, and some compensations appear, see Lemma \ref{lem:XY}. In fact, the heuristic below leads us to believe that the right asymptotics in the Brownian case should be $\varepsilon \left|\ln(\varepsilon)\right|^{n-1}$.
\end{remark}

\bigskip

The main part of the proof deals with the computation of an upper bound for the persistence probabilities. A simple approach would be to try to bound the quantity $\mathcal{S}_n$ by $\prod_{i=1}^n Z_{\theta_1}^{(i)}$ where $\theta_1$ is the value at which one of the L\'evy processes, say $Z^{(n)}$, reach its maximum on $[0,1]$. This yields of course two main difficulties.
\begin{enumerate}[$i)$]
\item First, the product of the other processes $\prod_{i=1}^{n-1} Z_{\theta_1}^{(i)}$ might not be positive. This can be however easily circumvented thanks to Slepian's inequality, since the processes are symmetric.
\item The second difficulty is less obvious and is due to the arcsine law for stable processes. There is a high probability that $\theta_1$ will be close to 0, hence, although $Z_{\theta_1}^{(n)}$ will be large, the remaining product $\prod_{i=1}^{n-1} Z_{\theta_1}^{(i)}$ will also be close to zero, thus not providing us with a good upper bound.
\end{enumerate}
The general idea of the proof will be to decompose the path of the processes $(Z^{(i)})$ at some last passage times and then use a time-reversal argument, so as to find a value not to close to the origin at which $Z^{(n)}$ is large enough. \\

\noindent
The outline of the paper is as follows : the large deviation results are proved in Section 2, the persistence probabilities in Section 3, and finally Section 4 provides the proof of an intermediary lemma.

\section{Large deviations}
The proof of the large deviation results relies on the symmetry of the processes $(Z^{(i)})$, and on the fact that the asymptotics of both random variables $|Z_1|$ and $\sup\limits_{0\leq u\leq 1} Z_u$ are similar. Indeed,  on the one hand, the lower bound is easily given by :  
$$\Pb(\mathcal{S}_n\geq x) \geq \Pb\left(\prod_{i=1}^n Z_1^{(i)}\geq x\right)=\frac{1}{2}\Pb\left(\prod_{i=1}^n |Z_1^{(i)}|\geq x\right).$$
On the other hand, still by symmetry, 
\begin{align*}
\Pb(\mathcal{S}_n\geq x)& \leq \Pb\left(\sup_{0\leq u\leq 1} Z_u^{(1)} \sup_{0\leq s\leq 1} \prod_{i=1}^{n-1} Z_s^{(i)} \geq x \right) +  \Pb\left(\inf_{0\leq u\leq 1} Z_u^{(1)} \inf_{0\leq s\leq 1} \prod_{i=1}^{n-1} Z_s^{(i)} \geq x \right) \\
&\leq 2\, \Pb\left(\sup_{0\leq u\leq 1} Z_u^{(1)} \sup_{0\leq s\leq 1} \prod_{i=1}^{n-1} Z_s^{(i)} \geq x \right) \\
&\leq 2^n\,  \Pb\left(\prod_{i=1}^n \sup_{0\leq u\leq 1} Z_u^{(i)}  \geq x \right) \qquad \qquad\text{(by iteration).}
\end{align*}
It remains thus to compute the involved quantities in both cases.\\

$\rightarrow$ In the Brownian case, since $\displaystyle \sup_{0\leq u\leq 1} W_u \= |W_1|$, we deduce that the asymptotics of $\mathcal{S}_n$ is given by that of $\prod_{i=1}^n |W_1^{(i)}|$. Its Mellin transform reads, for $\nu>-1$  :
\begin{equation}\label{eq:prodBM}
\E\left[\prod_{i=1}^n |W_1^{(i)}|^\nu  \right] = \left(\frac{2^\nu}{\pi}\right)^{\frac{n}{2}}\left(\Gamma\left(\frac{1+\nu}{2}\right)\right)^n. 
\end{equation}
The converse mapping theorem, see Janson \cite[Theorem 6.1]{Jan}, yields :
$$\Pb\left(\prod_{i=1}^n |W_1^{(i)}| \in dx\right)/dx \equi_{x\rightarrow +\infty} \kappa\, x^{\frac{1}{n}-1}\, e^{-\frac{n}{2}x^{\frac{2}{n}}}$$
for some positive constant $\kappa$. The result then follows by integration, using the asymptotics of the incomplete Gamma function.\\

$\rightarrow$ Next, when $\alpha\in (0,2)$,  it is known from Bertoin \cite[p.221]{Ber} that there exists $k>0$ such that 
$$\Pb(|L_1|\geq x) \equi_{x\rightarrow+\infty} \frac{2k}{x^{\alpha}}\qquad \text{ and } \qquad \Pb\left(\sup_{0\leq u\leq 1} L_u\geq x\right) \equi_{x\rightarrow+\infty}  \frac{k}{x^{\alpha}}.$$
Point 1. of Theorem \ref{theo:Levy} is then consequence of the following lemma (see for instance  Lemma 2  in Profeta-Simon \cite{PSWind}):
\begin{lemma}\label{lem:XY}
Let $X$ and $Y$ be two independent positive random variables satisfying the asymptotics :
$$\Pb(X\geq z)\equi_{z\rightarrow +\infty} a\frac{(\ln(z))^n}{z^\nu} \qquad \text{and}\qquad \Pb(Y\geq z)\equi_{z\rightarrow +\infty} b\frac{(\ln(z))^p}{z^\mu}$$
where $n,p\in \N$ and $a,b, \nu,  \mu$ are positive constants such that $0<\nu\leq \mu$. Then there exists $c>0$ such that :
$$\Pb(XY\geq z) \equi_{z\rightarrow +\infty} \begin{cases}
c\, z^{-\nu} (\ln(z))^n & \text{if } \nu<\mu\\
c\, z^{-\nu} (\ln(z))^{n+p+1}& \text{if }\nu=\mu.
\end{cases}$$
\qed
\end{lemma}

\section{Persistence probabilities}

We now turn our attention to the persistence estimates and start with some notations.  Let $X$ be a symmetric stable process. We denote by $\Pb_x$ the probability measure of $X$ when started from $x\in \R$, with the usual convention that $\Pb=\Pb_0$. Let $T_0$ be the first time that $X$ takes a negative value :
$$T_0=\inf\{t\geq0, X_t\leq 0\}.$$
We recall from Bertoin \cite[p.219]{Ber} that since $X$  is symmetric, there exists $c>0$ such that 
\begin{equation}\label{eq:limT0}
\Pb_{1}\left(T_{0} \geq t \right) \equi_{t\rightarrow +\infty}  \frac{c}{\sqrt{t}}.
\end{equation}
Finally, let  us introduce the last change of sign of $X$ before time $t>0$ : 
$$g_t = \sup\left\{0\leq u\leq t,\; X_u X_{u^-}\leq 0\right\}.$$
This random time will be the key to the computation of the persistence probabilities. 
\begin{remark}
In the following, when applying the Markov property, $\widehat{X}$ will always denote an independent copy of $X$. Besides, we shall use the notations $c$ and $\kappa$ to denote positive constants that may change from line to line.
\end{remark}

We first show that the asymptotics of the distribution of $g_1$ is similar to that of the arcsine law.
\begin{lemma}\label{lem:g}
There exists a positive constant $c$ such that 
$$\Pb(g_1\in dr) /dr \equi_{r\rightarrow 0} \frac{c}{\sqrt{r}}. $$
\end{lemma}
\begin{proof}
We first have, using the symmetry of $X$ and applying the Markov property with $r\in(0,1)$ :
$$
\Pb(g_1\leq r) = \E\left[  \widehat{\Pb}_{|X_r|}\left(\widehat{T}_0\geq 1-r\right)\right].
$$
By scaling, this is further equal to 
$$
\Pb(g_1\leq r) =\E\left[  \widehat{\Pb}_{1}\left(\widehat{T}_0\geq \frac{1-r}{r|X_1|^{\alpha}}\right)\right].
$$
Recall now from Doney-Savov \cite{DS} that under $\Pb_1$, the random variable $T_0$ admits a continuous density $h$ satisfying $h(z) \equi_{z\rightarrow +\infty} \kappa\, z^{-3/2}$ for some constant $\kappa>0$. Therefore, differentiating, we deduce that
$$\Pb(g_1\in dr) / dr = \frac{1}{r^2}\, \E\left[ \frac{1}{|X_1|^{\alpha}} h\left(  \frac{1-r}{r |X_1|^{\alpha}} \right)  \right] \equi_{r\rightarrow 0} \frac{\kappa}{\sqrt{r}}\, \E\left[ |X_1|^{\frac{\alpha}{2}}\right]
 $$
which is the announced result.
\end{proof}

\subsection{Lower bound for the persistence probabilities}
Observe first that by scaling 
\begin{align*}
\Pb(\mathcal{S}_n\leq \varepsilon) &= \Pb\left( \sup_{u\in[0, \varepsilon^{-\alpha/n}]}\prod_{i=1}^n Z_{u}^{(i)}  \leq 1\right)\\
&\geq \Pb\left( \sup_{u\in[0, \varepsilon^{-\alpha/n}]}\prod_{i=1}^n Z_{u}^{(i)}  \leq 1, \,\prod_{i=1}^n Z_{1}^{(i)}\leq 0, \, \sup_{1\leq i\leq n} g^{(i)}_{1/\varepsilon^{\frac{\alpha}{n}}} \leq 1  \right)\\
&= \Pb\left( \sup_{u\in[0, 1]}\prod_{i=1}^n Z_{u}^{(i)}  \leq 1, \,\prod_{i=1}^n Z_{1}^{(i)}\leq 0, \, \sup_{1\leq i\leq n} g^{(i)}_{1/\varepsilon^{\frac{\alpha}{n}}} \leq 1  \right)
\end{align*}
where the last equality follows from the fact that, by definition of the $(g^{(i)})$, the product $\prod_{i=1}^n Z^{(i)}$ remains negative after time 1. We now apply the Markov property at time 1 :
\begin{align}
\notag\Pb(\mathcal{S}_n\leq \varepsilon)& \geq  \E\left[ \prod_{i=1}^n\widehat{\Pb}_{|Z_1^{(i)}|}\left(\widehat{T}_0^{(i)}\geq \frac{1}{\varepsilon^{\frac{\alpha}{n}}}-1\right),
\sup_{u\in[0, 1]}\prod_{i=1}^n Z_{u}^{(i)}  \leq 1, \,\prod_{i=1}^n Z_{1}^{(i)}\leq 0  \right]\\
\label{eq:low}&\geq \E\left[ \prod_{i=1}^n\widehat{\Pb}_{1}\left( \widehat{T}_0^{(i)}\geq \frac{1}{\varepsilon^{\frac{\alpha}{n}}|Z_1^{(i)}|^\alpha} \right),
\sup_{u\in[0, 1]}\prod_{i=1}^n Z_{u}^{(i)}  \leq 1, \,\prod_{i=1}^n Z_{1}^{(i)}\leq 0  \right].
\end{align}
From (\ref{eq:limT0}), there exists $\kappa>0$ such that for $\delta>0$ small enough 
$$\widehat{\Pb}_{1}\left( \widehat{T}_0^{(i)}\geq \frac{1}{\varepsilon^{\frac{\alpha}{n}}|Z_1^{(i)}|^\alpha} \right)\mathds{1}_{\{\varepsilon^{\frac{\alpha}{n}}|Z_1^{(i)}|^\alpha\leq \delta\}} \geq \kappa\, \varepsilon^{\frac{\alpha}{2n}}|Z_1^{(i)}|^{\frac{\alpha}{2}}\mathds{1}_{\{\varepsilon^{\frac{\alpha}{n}}|Z_1^{(i)}|^\alpha\leq \delta\}}.$$
Plugging this inequality in (\ref{eq:low}), we deduce that 
\begin{align*}
\notag\Pb(\mathcal{S}_n\leq \varepsilon)&\geq   \kappa^n\, \varepsilon^{\frac{\alpha}{2}} \, \E\left[ \prod_{i=1}^n  |Z_1^{(i)}|^{\frac{\alpha}{2}}\mathds{1}_{\{\varepsilon^{\frac{\alpha}{n}}|Z_1^{(i)}|^\alpha\leq \delta\}},\,  \sup_{u\in[0, 1]}\prod_{i=1}^n Z_{u}^{(i)}  \leq 1, \,\prod_{i=1}^n Z_{1}^{(i)}\leq 0\right]\\
&\equi_{\varepsilon\rightarrow 0} \kappa^n \varepsilon^{\frac{\alpha}{2}}\,  \E\left[ \prod_{i=1}^n  |Z_1^{(i)}|^{\frac{\alpha}{2}} ,\,\sup_{u\in[0, 1]}\prod_{i=1}^n Z_{u}^{(i)}  \leq 1, \,\prod_{i=1}^n Z_{1}^{(i)}\leq 0\right]
\end{align*}
which gives the lower bound.

\subsection{Upper bound for the persistence probabilities}

Since all the processes $(Z^{(i)})$ have the same law, we first have :
$$
\Pb\left(\mathcal{S}_n \leq \varepsilon\right)   = n\, \Pb\left(\sup_{0\leq u\leq 1} \prod_{i=1}^n Z_u^{(i)} \leq \varepsilon,\; g_1^{(n)} \geq \sup_{1\leq i\leq n-1}g_1^{(i)}\right).
$$
To simplify the notation, we shall remove the superscript $^{(n)}$ and denote 
$$X=Z^{(n)},\qquad g_1=g_{1}^{(n)}\qquad \text{ and }\quad  \xi_t = \sup_{1\leq i\leq n-1}g_t^{(i)}.$$
This yields, with the usual convention that empty products equal 1, 
\begin{align*}
\Pb\left(\mathcal{S}_n \leq \varepsilon\right)   &= n\, \Pb\left(\sup_{0\leq u\leq 1} X_u \prod_{i=1}^{n-1} Z_u^{(i)} \leq \varepsilon,\; g_1 \geq\xi_1\right)\\
&\leq n\, \Pb\left(\sup_{0\leq u< g_1}X_u \prod_{i=1}^{n-1} Z_{u}^{(i)} \leq \varepsilon,\; g_1 \geq\xi_1\right)\\
&=2 n\, \Pb\left(\sup_{0\leq u< 1}X_{ug_1}\prod_{i=1}^{n-1} Z_{ug_1}^{(i)} \leq \varepsilon,  \;  \prod_{i=1}^{n-1} Z_{g_1}^{(i)} \geq0,\; g_1\geq \xi_1\right)
\end{align*}
where the last equality follows by symmetry.  
By scaling, we further obtain 
$$
\Pb\left(\mathcal{S}_n \leq \varepsilon\right) \leq 2 n\,\int_0^1  \Pb\left(\sup_{0\leq u< 1} \frac{X_{ur}}{r^{1/\alpha}} \prod_{i=1}^{n-1} Z_{u}^{(i)}  r^{\frac{n}{\alpha}} \leq \varepsilon,  \; \prod_{i=1}^{n-1} Z_{1}^{(i)}\geq0,\;1 \geq \xi_{\frac{1}{r}} \bigg|\; g_1=r  \right)\Pb(g_1\in dr).$$
We set $X^{(x,t,y)}$ for the $\alpha$-stable bridge of length $t$ starting from $x$ and ending at $y$. Notice that when $X=W$ is a Brownian motion, then $g_1$ coincides with the last zero of $W$ before time 1, so that $W_{g_1}=0$ a.s. and it is  well-known that the process $\left(\frac{W_{ug_1}}{\sqrt{g_1}}, \, 0\leq u\leq 1\right)$ is a standard Brownian bridge, independent from $g_1$, see Bertoin \cite[p.230]{Ber}. We shall extend this result to the stable case in the following lemma, whose proof is postponed at the end of the paper.
\begin{lemma}\label{lem}
We set by convention $X_{0^-}=X_0$. Conditionally on the event $\left\{\frac{X_{g_1^-}}{g_1^{1/\alpha}}=a\right\}$, the process
$$\left(\frac{X_{(ug_1)^-}}{g_1^{1/\alpha}},\; 0\leq u\leq 1\right)
$$ 
is independent from $g_1$ and has the same law as the stable bridge $\left(X^{(0,1,a)}_{u^{-}}, 0\leq u\leq 1\right)$. 
\end{lemma}

Let us denote  by $\rho(da,dr)$ the law of the pair $\left(g_1^{-1/\alpha} X_{g_1^-},\; g_1\right)$. Since the $(Z^{(i)})$ are quasi-left continuous and independent from $X$,  we deduce from Lemma \ref{lem} that 
\begin{align}
\Pb\left(\mathcal{S}_n \leq \varepsilon\right)
\notag &\leq 2 n\, \int_0^{+\infty} \int_0^1  \Pb\left(\sup_{0\leq u\leq 1} X_{u^-}^{(0,1,a)} \prod_{i=1}^{n-1} Z_{u}^{(i)} r^{\frac{n}{\alpha}} \leq \varepsilon,  \; \prod_{i=1}^{n-1} Z_{1}^{(i)}\geq0,\;1 \geq \xi_{\frac{1}{r}} \right)\rho(da,dr)\\
\label{eq:persup}&= 2 n\,  \int_0^{+\infty} \int_0^1  \Pb\left(\sup_{0\leq u\leq 1} X_{u}^{(a,1,0)} \prod_{i=1}^{n-1} Z_{1-u}^{(i)} r^{\frac{n}{\alpha}} \leq \varepsilon,  \; \prod_{i=1}^{n-1} Z_{1}^{(i)}\geq0,\;1 \geq \xi_{\frac{1}{r}} \right)\rho(da,dr) 
\end{align}
where the equality follows from the time-reversal property of stable bridges. We shall now decompose the right-hand side of this inequality according as $\{a\leq1\}$ or  $\{a>1\}$.

\subsubsection{The case $\{a\leq 1\}$}
We start with the term giving the main contribution. Let us denote by $p_t$ the density of the random variable $X_t$, and recall that it is even, and decreasing on $(0,+\infty)$. Using the absolute continuity formula of the stable bridge, we get :
 \begin{align}
\notag &\iint_0^1  \Pb\left(\sup_{0\leq u\leq 1} X_{u}^{(a,1,0)} \prod_{i=1}^{n-1} Z_{1-u}^{(i)} r^{\frac{n}{\alpha}} \leq \varepsilon,  \; \prod_{i=1}^{n-1} Z_{1}^{(i)}\geq0,\;1 \geq \xi_{\frac{1}{r}} \right)\rho(da,dr)  \\
\notag &\qquad \leq \iint_0^1 \E\left[\frac{p_{\frac{1}{2}}(a+X_{\frac{1}{2}})}{p_1(a)}    ,\,\sup_{0\leq u\leq 1/2}(a+ X_u)\prod_{i=1}^{n-1} Z_{1-u}^{(i)}r^{\frac{n}{\alpha}} \leq \varepsilon,  \; \prod_{i=1}^{n-1} Z_{1}^{(i)}\geq0,\;1 \geq \xi_{\frac{1}{r}} \right]\rho(da,dr) \\
\label{eq:In}&\qquad \leq \frac{p_{\frac{1}{2}}(0)}{p_1(1)} \,\iint_0^1  \Pb\left(\sup_{0\leq u\leq 1/2}(a+ X_u)\prod_{i=1}^{n-1} Z_{1-u}^{(i)}r^{\frac{n}{\alpha}} \leq \varepsilon,  \; \prod_{i=1}^{n-1} Z_{1}^{(i)}\geq0,\;1 \geq\xi_{\frac{1}{r}} \right)\rho(da,dr).
\end{align} 
We now study the integrand in (\ref{eq:In}). Recall that $X$ admits the representation $(B_{\tau_u}, u\geq0)$ where $B$ is a standard Brownian motion and $\tau$ a stable subordinator with index $\frac{\alpha}{2}$ independent from $B$. Let us consider the conditional expectation :
$$\Pb\left(\sup_{0\leq u\leq 1/2} (a+B_{\eta_u}) \prod_{i=1}^{n-1} \omega_{1-u}^{(i)} r^{\frac{n}{\alpha}} \leq \varepsilon\bigg|  \begin{array}{cc}
Z^{(i)}=\omega^{(i)},&\; \tau=\eta\\
1\leq i\leq n-1&
\end{array}  \right) $$
where $\eta$ and $(\omega^{(i)}, 1\leq i \leq n-1)$  are some fixed c\`adl\`ag paths. We apply Slepian's lemma with the Gaussian processes
$$U_u =z\prod_{i=1}^{n-1} \omega_{u}^{(i)}+B_{\eta_u} \prod_{i=1}^{n-1} \omega_{u}^{(i)} \qquad \text{ and } \qquad V_u= a\prod_{i=1}^{n-1} \omega_{u}^{(i)}+B_{\eta_u} \prod_{i=1}^{n-1} |\omega_{u}^{(i)}|$$
which satisfy for every $0\leq u\leq s\leq \frac{1}{2}$, 
$$\E[U_u]= \E[V_u],\qquad \E\left[U_u^2\right]= \E\left[V_u^2\right] \quad\text{and}\quad \E[U_uU_s]\leq \E[V_uV_s].$$
This yields, using the tower property of conditional expectations :
\begin{multline*}
\Pb\left(\sup_{0\leq u\leq 1/2} (a+X_{u})\prod_{i=1}^{n-1} Z_{1-u}^{(i)}r^{\frac{n}{\alpha}} \leq \varepsilon, \; \prod_{i=1}^{n-1} Z_{1}^{(i)} \geq0,\;1 \geq \xi_{\frac{1}{r}}  \right) \\
\leq \Pb\left(\sup_{0\leq u\leq 1/2} \left(a\prod_{i=1}^{n-1} Z_{1-u}^{(i)}+ X_{u} \prod_{i=1}^{n-1} |Z_{1-u}^{(i)}| \right)r^{\frac{n}{\alpha}} \leq \varepsilon,\;\prod_{i=1}^{n-1} Z_{1}^{(i)} \geq0, \;1 \geq \xi_{\frac{1}{r}}  \right).
\end{multline*}
Observe next that, by taking $u=0$, this quantity is null as soon as $a\prod_{i=1}^{n-1} Z_{1}^{(i)} r^{\frac{n}{\alpha}}\geq \varepsilon$. Therefore, denoting  $\theta_{\frac{1}{2}} = \mathop{\text{Argmax}}\limits_{0\leq u\leq 1/2}X_u$, we  may replace the supremum by its value at $\theta_{\frac{1}{2}}$ to get the bound
\begin{equation}\label{eq:boundtheta}
\Pb\left( \bigg(a\prod_{i=1}^{n-1} Z_{1-\theta_{\frac{1}{2}}}^{(i)}+ X_{\theta_{\frac{1}{2}}} \prod_{i=1}^{n-1} |Z_{1-\theta_{\frac{1}{2}}}^{(i)}| \bigg)r^{\frac{n}{\alpha}} \leq \varepsilon,\;  \varepsilon\geq a\, r^{\frac{n}{\alpha}} \prod_{i=1}^{n-1} Z_{1}^{(i)} \geq0, \;1 \geq \xi_{\frac{1}{r}}  \right).
\end{equation}
We further decompose this integral according to the sign of $\prod_{i=1}^{n-1} Z_{1-\theta_{\frac{1}{2}}}^{(i)}$.  
\begin{enumerate}[$i)$]
\item When $\prod_{i=1}^{n-1} Z_{1-\theta_{\frac{1}{2}}}^{(i)}\geq0$, the expression (\ref{eq:boundtheta}) is smaller than 
\begin{equation}\label{eq:Hn+}
\Pb\left( X_{\theta_{\frac{1}{2}}} \prod_{i=1}^{n-1} |Z_{1-\theta_{\frac{1}{2}}}^{(i)}| r^{\frac{n}{\alpha}} \leq \varepsilon,\;1 \geq \xi_{\frac{1}{r}}     \right)=:I_n(r, \varepsilon) . 
\end{equation}
\item When $\prod_{i=1}^{n-1} Z_{1-\theta_{\frac{1}{2}}}^{(i)}\leq0$, the situation is slightly more complex. We have
\begin{align*}
 &    \Pb\left( (X_{\theta_{\frac{1}{2}}}-a) \prod_{i=1}^{n-1} |Z_{1-\theta_{\frac{1}{2}}}^{(i)}|r^{\frac{n}{\alpha}} \leq \varepsilon,\;  \varepsilon \geq a\,r^{\frac{n}{\alpha}}\prod_{i=1}^{n-1} Z_{1}^{(i)} \geq0,\;1 \geq \xi_{\frac{1}{r}} \right)  \\
& \qquad\quad \leq      \Pb\left(X_{\theta_{\frac{1}{2}}} \prod_{i=1}^{n-1} |Z_{1-\theta_{\frac{1}{2}}}^{(i)}| r^{\frac{n}{\alpha}} \leq \varepsilon\left(1 +\frac{\prod_{i=1}^{n-1} |Z_{1-\theta_{\frac{1}{2}}}^{(i)}|}{\prod_{i=1}^{n-1} Z_{1}^{(i)}}   \right)  ,\prod_{i=1}^{n-1} Z_{1}^{(i)} \geq0,\;1 \geq \xi_{\frac{1}{r}}  \right)   \\
 &\qquad \qquad\leq I_n(r,2\varepsilon) + \Pb\left(X_{\theta_{\frac{1}{2}}}\prod_{i=1}^{n-1} |Z_{1}^{(i)}| r^{\frac{n}{\alpha}} \leq 2\varepsilon,\;1 \geq \xi_{\frac{1}{r}}   \right)=: I_n(r,2 \varepsilon) + J_n(r,2 \varepsilon)
\end{align*}
\end{enumerate}
where we have used in the last line the inequality : $1_{\{x\leq \varepsilon(a+b)\}} \leq 1_{\{x\leq 2a\varepsilon\}} + 1_{\{x\leq 2b\varepsilon\}}$. Going back to (\ref{eq:In}), we are thus led to study the asymptotics of
$$  \int_0^1  (I_n(r,\varepsilon)+J_n(r,\varepsilon))\Pb(g_1\in dr).$$
We start with  $I_n(r,\varepsilon)$ which will give the main contribution. From Lemma \ref{lem:g}, we may choose 
$\delta \in(0,1)$ small enough such that 
\begin{equation}\label{eq:infg}
\forall r\leq \delta,\qquad \Pb(g_1\in dr)/dr \leq \frac{c}{\sqrt{r}}
\end{equation}
for some constant $c>0$.  On the one hand, when $r\geq \delta$, we obtain since $\theta_{\frac{1}{2}}\leq \frac{1}{2}$ :
$$
\int_{\delta}^1 I_n(r, \varepsilon)\Pb\left(g_1\in dr   \right)\leq \Pb\left( X_{\theta_{\frac{1}{2}}} \prod_{i=1}^{n-1} |Z_{1}^{(i)}| \delta^{n/\alpha}  \leq 2^{\frac{n-1}{\alpha}}  \varepsilon    \right)
 \equi_{\varepsilon\rightarrow +\infty} \begin{cases}
   \kappa \, \varepsilon^{\frac{\alpha}{2}} \quad \text{if } \alpha\in(0,2),\\
   \kappa \, \varepsilon \left|\ln(\varepsilon)\right|^{n-1} \quad \text{if } \alpha=2.\\
\end{cases}
$$
On the other hand, when $r\leq \delta$, we deduce from the  Markov property at time 1 that  :
\begin{equation*}
I_n(r,\varepsilon) \leq \E\left[ \prod_{i=1}^{n-1}\widehat{\Pb}^{(i)}_{1}\left(|Z_1^{(i)}|^\alpha\widehat{T}_{0}^{(i)} \geq\frac{1-\delta}{r}\right),X_{\theta_{\frac{1}{2}}} \prod_{i=1}^{n-1} |Z_{1-\theta_{\frac{1}{2}}}^{(i)}| r^{\frac{n}{\alpha}} \leq \varepsilon\right].
\end{equation*}
Using the identity $Z_1^{(i)} \=Z_{1-\theta_{\frac{1}{2}}}^{(i)} +\theta_{\frac{1}{2}}^{\frac{1}{\alpha}}  Y_{1}^{(i)}$ where $Y_1^{(i)}$ is a copy of $Z_1^{(i)}$,  independent from the processes $(Z^{(i)})$ and $(\widehat{Z}^{(i)})$, we then obtain the bound \begin{equation}\label{eq:borneIn}
I_n(r,\varepsilon) \leq\E\left[  \prod_{i=1}^{n-1}\widehat{\Pb}^{(i)}_{1}\left( \Big( |Z_{1-\theta_{\frac{1}{2}}}^{(i)}|^\alpha  + |Y_1^{(i)}|^\alpha\Big)\widehat{T}_{0}^{(i)} \geq\frac{1-\delta}{ 2r}\right),\, X_{\theta_{\frac{1}{2}}} \prod_{i=1}^{n-1} |Z_{1-\theta_{\frac{1}{2}}}^{(i)}|  r^{\frac{n}{\alpha}} \leq  1     \right]
 \end{equation}
 where we have used the classic inequality $|x+y|^\alpha \leq 2 (|x|^{\alpha} + |y|^{\alpha})$ since $\alpha\in(0,2]$.
We further assume that $\delta$ is taken small enough so that, from Lemma \ref{lem:XY} and the asymptotics (\ref{eq:limT0}), we have  
 \begin{equation}\label{eq:limY}
 \forall r\leq \delta, \qquad \E\left[\widehat{\Pb}^{(i)}_{1}\left( \Big( |Z_{1-\theta_{\frac{1}{2}}}^{(i)}|^\alpha\vee 1  + |Y_1^{(i)}|^\alpha\Big)\widehat{T}_{0}^{(i)} \geq\frac{1-\delta}{ 2r}\right)\right] \leq \kappa \sqrt{r}
 \end{equation}
 for some positive constant $\kappa$, and where $a\vee b = \max(a,b)$.  We shall now proceed by iteration.
\begin{enumerate}[$i)$]
\item If $|Z_{1-\theta_{\frac{1}{2}}}^{(n-1)}| \geq 1 $, then, we may remove $|Z_{1-\theta_{\frac{1}{2}}}^{(n-1)}|$ from the second product in (\ref{eq:borneIn}), and deduce from (\ref{eq:limY}) that $I_n(r,\varepsilon)$ is smaller than  
$$ \kappa\,\sqrt{r}\,  \E\left[  \prod_{i=1}^{n-2}\widehat{\Pb}^{(i)}_{1}\left( \Big( |Z_{1-\theta_{\frac{1}{2}}}^{(i)}|^\alpha  + |Y_1^{(i)}|^\alpha\Big)\widehat{T}_{0}^{(i)} \geq\frac{1-\delta}{ 2r}\right),\, X_{\theta_{\frac{1}{2}}} \prod_{i=1}^{n-2} |Z_{1-\theta_{\frac{1}{2}}}^{(i)}|  r^{\frac{n}{\alpha}} \leq  1     \right]. 
$$
\item If $|Z_{1-\theta_{\frac{1}{2}}}^{(n-1)}| \leq 1 $, then, we may replace $|Z_{1-\theta_{\frac{1}{2}}}^{(n-1)}|$ by 1 in the first product in (\ref{eq:borneIn}), and deduce, still from (\ref{eq:limY}), that $I_n(r,\varepsilon)$ is smaller than $$\kappa\,\sqrt{r} \, \E\left[  \prod_{i=1}^{n-2}\widehat{\Pb}^{(i)}_{1}\left( \Big( |Z_{1-\theta_{\frac{1}{2}}}^{(i)}|^\alpha  + |Y_1^{(i)}|^\alpha\Big)\widehat{T}_{0}^{(i)} \geq\frac{1-\delta}{ 2r}\right),\, X_{\theta_{\frac{1}{2}}} \prod_{i=1}^{n-1} |Z_{1-\theta_{\frac{1}{2}}}^{(i)}|  r^{\frac{n}{\alpha}} \leq  1     \right]. $$
\end{enumerate}
Iterating the procedure, we obtain that $I_n(r,\varepsilon)$ may be bounded by a sum of $2^{n-1}$ terms :
$$I_n(r,\varepsilon) \leq  \kappa\, r^{\frac{n-1}{2}}  \sum_{\Delta\subset \{1,\ldots, n-1\} }\Pb\left(X_{\theta_{\frac{1}{2}}}    \prod_{i\in\Delta} |Z_{1-\theta_{\frac{1}{2}}}^{(i)} |r^{\frac{n}{\alpha}} \leq 1 \right)
$$
where the sum is taken over all the  subsets of  $\{1,2,\ldots, n-1\}$ (including the empty set). The change of variable $\varepsilon x = r^{\frac{n}{\alpha}}$ and the estimate (\ref{eq:infg}) yield then the upper bound 
\begin{equation}\label{eq:InD}
\int_0^{\delta}I_n(r,\varepsilon) \Pb(g_1\in dr) \leq  \kappa \,\varepsilon^{\frac{\alpha}{2}}  \sum_{\Delta\subset \{1,\ldots, n-1\} }\int_0^{\frac{\delta^{n/\alpha}}{\varepsilon}} 
 \Pb\left( X_{\theta_{\frac{1}{2}}} \prod_{i\in \Delta} |Z_{1}^{(i)}|  x \leq  2^{\frac{n-1}{\alpha}}     \right)
x^{\frac{\alpha}{2}-1}dx
\end{equation}
and it remains to study the asymptotics of the integrands. From Lemma \ref{lem:XY} and (\ref{eq:asymp}), we deduce that
\begin{enumerate}[$\rightarrow$]
\item when $\alpha\in (0,2)$ all the terms have the same contribution :
$$\Pb\left( X_{\theta_{\frac{1}{2}}} \prod_{i\in \Delta} |L_{1}^{(i)}|  x \leq  2^{\frac{n-1}{\alpha}}     \right)\equi_{x\rightarrow+\infty} c \left(\frac{1}{x}\right)^{\alpha/2}
$$
\item while, for $\alpha=2$, they depend on the cardinal of $\Delta$ :
$$ \Pb\left( X_{\theta_{\frac{1}{2}}} \prod_{i\in \Delta} |W_{1}^{(i)}|  x \leq  2^{\frac{n-1}{2}}     \right)\equi_{x\rightarrow+\infty} c\, \frac{(\ln(x))^{|\Delta|}}{x}.
$$
\end{enumerate}
Plugging these expressions in (\ref{eq:InD}) finally gives the announced upper bound.\\

The study of the asymptotics of $J_n(r,\varepsilon)$ follows the same pattern of proof, except that we do not need to introduce the random variables $(Y_1^{(i)})$. Indeed, when $r\geq \delta$, we get the same asymptotics bound while  for $r\leq \delta$   we obtain, applying the Markov property :
$$J_n(r,\varepsilon)\leq   \E\left[  \prod_{i=1}^{n-1}\widehat{\Pb}^{(i)}_{1}\left(|Z_1^{(i)}|^\alpha\widehat{T}_{0}^{(i)} \geq\frac{1-\delta}{ r}\right),\, X_{\theta_{\frac{1}{2}}} \prod_{i=1}^{n-1} |Z_{1}^{(i)}|  r^{\frac{n}{\alpha}} \leq  1     \right].
$$
Using the decompositions $|Z_1^{(i)}| \leq 1$ (resp. $|Z_1^{(i)}| \geq 1$) and following the same steps as for $I_n(r,\varepsilon)$, we deduce that  
\begin{multline}\label{eq:Jn}
\int_0^{\delta }J_n(r,\varepsilon) \Pb(g_1\in dr)\\
\leq \kappa\,  \varepsilon^{\frac{\alpha}{2}}  \sum_{\Delta\subset \{1,\ldots,n-1\}} \int_0^{\frac{\delta^{n/\alpha}}{\varepsilon}} 
   \E\left[ \prod_{i\in \Delta} |Z_1^{(i)}|^{\frac{\alpha}{2}},\;   X_{\theta_{\frac{1}{2}}} \prod_{i\in \Delta} |Z_{1}^{(i)}|  x \leq  2^{\frac{n-1}{\alpha}}     \right]
x^{\frac{\alpha}{2}-1}dx
\end{multline}
and, as $\varepsilon\rightarrow 0$, all the terms on the right-hand side have the same asymptotics :   $\varepsilon^{\frac{\alpha}{2}} \left|\ln(\varepsilon)\right|$.\\

 \subsubsection{The case $\{a \geq 1\}$}
Starting back from (\ref{eq:persup}), we first bound the supremum by its value at $u=0$~:
\begin{align*}
&\int_1^{+\infty} \int_0^1  \Pb\left(\sup_{0\leq u\leq 1} X_{u}^{(a,1,0)} \prod_{i=1}^{n-1} Z_{1-u}^{(i)} r^{\frac{n}{\alpha}} \leq \varepsilon,  \; \prod_{i=1}^{n-1} Z_{1}^{(i)}\geq0,\;1 \geq \xi_{\frac{1}{r}} \right)\rho(da,dr)  \\
&\qquad \qquad\leq\int_1^{+\infty} \int_0^1  \Pb\left(a \prod_{i=1}^{n-1} Z_{1}^{(i)} r^{\frac{n}{\alpha}} \leq \varepsilon,  \; \prod_{i=1}^{n-1} Z_{1}^{(i)}\geq0,\;1 \geq \xi_{\frac{1}{r}} \right)\rho(da,dr) \\
&\qquad\qquad \leq \int_0^1  \Pb\left( \prod_{i=1}^{n-1} |Z_{1}^{(i)}| r^{\frac{n}{\alpha}} \leq \varepsilon,\;1 \geq \xi_{\frac{1}{r}} \right)\Pb\left(g_1\in dr   \right). 
\end{align*} 
The study of this last expression will be similar to that of $J_n(r,\varepsilon)$, replacing $X_{\theta_{\frac{1}{2}}}$ by 1. Indeed, on the one hand, we first deduce from Lemma \ref{lem:XY}, taking $\delta$ small enough as before, that :
\begin{equation*}
 \int_{\delta}^1  \Pb\left( \prod_{i=1}^{n-1} |Z_{1}^{(i)}| r^{\frac{n}{\alpha}} \leq \varepsilon, \;1 \geq \xi_{\frac{1}{r}} \right)\Pb\left(g_1\in dr   \right) 
 \leq \Pb\left( \prod_{i=1}^{n-1} |Z_{1}^{(i)}|  \delta^{\frac{n}{\alpha}} \leq  \varepsilon \right)\equi_{\varepsilon\rightarrow 0}  \kappa\,  \varepsilon \left|\ln(\varepsilon)\right|^{n-2}.
\end{equation*}
On the other hand, for $r\leq \delta$, we deduce, as for (\ref{eq:Jn}), that 
\begin{multline*}
\int_0^{\delta}  \Pb\left( \prod_{i=1}^{n-1} |Z_{1}^{(i)}| r^{\frac{n}{\alpha}} \leq \varepsilon,\;1 \geq \xi_{\frac{1}{r}} \right)\Pb\left(g_1\in dr   \right) \\
\leq  \kappa\,  \varepsilon^{\frac{\alpha}{2}}  \sum_{\Delta\subset \{1,\ldots,n-1\}} \int_0^{\frac{\delta^{n/\alpha}}{\varepsilon}} 
   \E\left[ \prod_{i\in \Delta} |Z_1^{(i)}|^{\frac{\alpha}{2}},\;   \prod_{i\in \Delta} |Z_{1}^{(i)}|  x \leq  2^{\frac{n-1}{\alpha}}     \right]
x^{\frac{\alpha}{2}-1}dx.\end{multline*}
When $\varepsilon\rightarrow 0$, all the integrals on the right-hand side are finite, hence we obtain the asymptotics  $\varepsilon^{\frac{\alpha}{2}}$
which is  negligible. \qed

\section{Proof of Lemma \ref{lem}}

\begin{proof}
This lemma being classic for Brownian motion, we assume that $\alpha\in(0,2)$. Let $0< s\leq t\leq 1$ and take $F$ a positive functional. Let us denote by $f(y; z,r)$ the probability density function of $(X_{T_0}, T_0)$ when $X_0=y$. By symmetry and time reversal, we first have 
\begin{multline*}
\E\left[F\left(\frac{X_{u^-}}{g_1^{1/\alpha}}, s\leq u\leq g_1\right) 1_{\{g_1\geq t\}}   \right] \\
=2\int_0^{+\infty} \E^{(y,1,0)}\left[F\left(\frac{X_{u}}{(1-T_0)^{1/\alpha}}, T_0\leq u\leq 1-s \right) 1_{\{T_0\leq 1-t\}} \right] p_1(y) dy.
\end{multline*}
The absolute continuity formula for the stable bridge as well as the Markov property then yield
\begin{align*}
&2\int_0^{+\infty}\E_y\left[p_s(X_{1-s})  F\left(\frac{X_{u}}{(1-T_0)^{1/\alpha}}, T_0\leq u\leq 1-s \right) 1_{\{T_0\leq 1-t\}} \right]  dy\\
&= 2\int_0^{+\infty} \int_{-\infty}^0 \int_0^{1-t}  \E_z\left[p_s(X_{1-s-r})  F\left(\frac{X_{u}}{(1-r)^{1/\alpha}}, 0\leq u\leq 1-s-r \right) \right]  f(y; z,r) dydzdr
\end{align*}
Next, by scaling and using that $t^{1/\alpha}p_t(z) = p_1\left(\frac{z}{t^{1/\alpha}}\right)$, 
\begin{align*}
&2\int_0^{+\infty} \int_{-\infty}^0 \int_0^{1-t}  \E_{\frac{z}{(1-r)^{1/\alpha}}}\left[p_s((1-r)^{1/\alpha}X_{1-\frac{s}{1-r}})  F\left(X_{\frac{u}{1-r}}, 0\leq u\leq 1-s-r \right) \right]  f(y; z,r) dydzdr\\
&\quad= 2 \int_0^{+\infty} \int_{-\infty}^0\int_0^{1-t}  \E_{a}\left[p_{\frac{s}{1-r}}(X_{1-\frac{s}{1-r}})  F\left(X_{u}, 0\leq u\leq 1-\frac{s}{1-r} \right) \right]  f(y; a (1-r)^{1/\alpha},r) dy da dr\\
&\quad= 2 \int_0^{+\infty}\int_{-\infty}^0 \int_0^{1-t}  \E^{(a,1,0)}\left[F\left(X_{u}, 0\leq u\leq 1-\frac{s}{1-r} \right) \right]p_1(a)  f(y; a (1-r)^{1/\alpha},r) dy  da dr\\
&\quad= 2 \int_{-\infty}^0 da\, \E^{(0,1,a)}\left[F\left(X_{u^-}, \frac{s}{1-r}\leq u\leq 1 \right) \right]p_1(a)  \int_0^{+\infty}  \int_0^{1-t} f(y; a (1-r)^{1/\alpha},r) dy  dr
\end{align*}
where in the second line we have used the change of variable $z=a(1-r)^{1/\alpha}$. 
Letting $s\rightarrow0$, we finally deduce that 
\begin{multline*}
\E\left[F\left(\frac{X_{u^-}}{g_1^{1/\alpha}}, 0\leq u\leq g_1\right) 1_{\{g_1\geq t\}}  \right]\\
=2 \int_{-\infty}^{0} da\,\E^{(0,1,a)}\left[F\left(X_{u^-},0\leq u\leq 1 \right) \right]p_1(a)  \int_0^{+\infty}  \int_0^{1-t} f(y; a (1-r)^{1/\alpha},r) dy  dr\end{multline*}
which proves Lemma \ref{lem}.

\end{proof}

\end{document}